\documentclass[reqno, 12pt]{amsart}

\usepackage{amsfonts, amsthm, amsmath, amssymb}

\usepackage{a4wide}

\usepackage{hyperref}

\RequirePackage{mathrsfs} \let\mathcal\mathscr

\numberwithin{equation}{section}

\newtheorem{theorem}{Theorem}[section]

\newtheorem{proposition}[theorem]{Proposition}

\theoremstyle{definition}
\newtheorem*{ack}{Acknowledgements}

\newtheorem{definition}[theorem]{Definition}

\renewcommand{\phi}{\varphi}
\renewcommand{\rho}{\varrho}

\newcommand{\PP}{\mathbb{P}}
\renewcommand{\AA}{\mathbb{A}}
\newcommand{\A}{\mathbf{A}}
\newcommand{\FF}{\mathbb{F}}
\newcommand{\ZZ}{\mathbb{Z}}

\newcommand{\NN}{\mathbb{Z}_{>0}}
\newcommand{\QQ}{\mathbb{Q}}
\newcommand{\RR}{\mathbb{R}}
\newcommand{\CC}{\mathbb{C}}

\newcommand{\val}{{\rm val}}

\renewcommand{\leq}{\leqslant}

\renewcommand{\geq}{\geqslant}

\renewcommand{\bar}{\overline}

\newcommand{\x}{\mathbf{x}}
\newcommand{\y}{\mathbf{y}}
\renewcommand{\c}{\mathbf{c}}

\renewcommand{\u}{\mathbf{u}}

\renewcommand{\b}{\mathbf{b}}
\renewcommand{\a}{\mathbf{a}}

\newcommand{\la}{\lambda}

\newcommand{\ve}{\varepsilon}

\newcommand{\ee}{\varepsilon}

\newcommand{\bla}{{\boldsymbol{\lambda}}}

\DeclareMathOperator{\Pic}{Pic}

\DeclareMathOperator{\meas}{meas}

\DeclareMathOperator{\sign}{sign}

\DeclareMathOperator{\Mod}{mod} 
\renewcommand{\bmod}[1]{\,(\Mod{#1})}

\newcommand{\Br}{{\rm Br}}
\newcommand{\W}{\mathcal{W}}
\newcommand{\V}{\mathcal{V}}
\newcommand{\U}{\mathcal{U}}

\renewcommand{\=}{\equiv}
\renewcommand{\t}{\mathbf{t}}

\begin{document}
\title[Rational points on pencils of conics and quadrics]{Rational points on pencils of conics and quadrics\\ with many degenerate fibres}

\author{T.D. Browning}
\address{School of Mathematics\\
University of Bristol\\ Bristol\\ BS8 1TW\\U.K.}
\email{t.d.browning@bristol.ac.uk}

\author{L. Matthiesen}
\address{D\'epartement de Math\'ematiques\\
Universit\'e Paris-Sud\\ 91405 Orsay\\France}
\email{lilian.matthiesen@math.u-psud.fr}

\author{A.N. Skorobogatov}
\address{
Department of Mathematics\\ Imperial College London \\ London\\ SW7 2BZ\\U.K.}
\email{a.skorobogatov@imperial.ac.uk}


\begin{abstract}
For any pencil of conics or higher-dimensional quadrics over $\QQ$, 
with all degenerate fibres defined over $\QQ$, we
show that the Brauer--Manin obstruction controls 
weak approximation. 
The proof is based on the Hasse principle and weak approximation 
for some special intersections of quadrics over $\QQ$, which is
a consequence of recent advances in additive combinatorics.
\end{abstract}

\maketitle

\section{Introduction}
\label{s:intro}

The arithmetic of conic bundle surfaces $X$ over number fields $k$
has long been the object of intensive study. 
Such varieties are defined to be projective non-singular surfaces $X$
over $k$, which are equipped with a dominant $k$-morphism
$\pi:X\rightarrow \PP_k^1$, all of whose fibres are conics.
Colliot-Th\'el\`ene and Sansuc \cite{ct-s-79} conjectured in 1979 
that the Brauer--Manin obstruction is the only 
obstruction to the Hasse principle and weak approximation for 
conic bundle surfaces. They also showed in \cite{ct-s-82} how one may 
use Schinzel's hypothesis (also known as the hypothesis of 
Bouniakowsky, Dickson and Schinzel) to study the Hasse principle, 
weak approximation and the Brauer--Manin obstruction on conic bundles 
over $\QQ$ given by an equation $x^2-ay^2=P(t)$, 
with $a \in \QQ^*$ and $P(t)$ a polynomial of arbitrary degree.
Later the same method was taken up again and generalised by 
Serre \cite[Ch.~II, Annexe]{serre} and by Swinnerton-Dyer \cite{swd-94},
general results for pencils of Severi--Brauer varieties or 2-dimensional quadrics
being given by Colliot-Th\'el\`ene and Swinnerton-Dyer in 
\cite{ct-swd-94}.

To discuss unconditional resolutions of the conjecture it is convenient
to assume without loss of generality
that the conic bundle $\pi:X\to\PP^1_k$ is relatively minimal,
which means that no irreducible component of a degenerate 
fibre is defined over the field of definition of that fibre.
The work to date has been restricted to the case in which the number 
 $r$ of  degenerate  fibres of $\pi$ is small.
When $0\leq r\leq 5$ everything is known about the qualitative 
arithmetic unconditionally.  Thus for $0\leq r\leq 3$ 
the Hasse  principle holds and, furthermore, 
$X$ is $k$-rational as soon as $X(k)\neq \emptyset$.
When $r=4$ one finds that  $X$ is either 
a Ch\^atelet surface or else
a  quartic del Pezzo surface with a conic bundle structure. 
The former is handled by Colliot-Th\'el\`ene, Sansuc and 
Swinnerton-Dyer \cite{CT} and the latter by 
Colliot-Th\'el\`ene \cite{ct-4} and Salberger \cite{s-conic}, using 
independent approaches.
When $r=5$,  $X$ is $k$-isomorphic to a smooth
cubic surface containing a line defined over $k$. 
In particular $X(k)\neq \emptyset$. On contracting the line to obtain 
a del Pezzo surface of degree $4$ with a $k$-point, work of 
Salberger and Skorobogatov \cite{salb-skoro} ensures that 
the Brauer--Manin obstruction is the only obstruction to 
weak approximation in this case.
Finally, when $r=6$, some special cases have been dealt with by 
Swinnerton-Dyer \cite{swd-toulouse} (cf.\ \cite[\S 7.4]{skoro}).

In this article we demonstrate how recent advances in additive 
combinatorics 
can help deal unconditionally with conic bundle surfaces $X$ 
defined over $\QQ$, in which the degenerate fibres are 
all defined over $\QQ$. The point where two components of a degenerate
fibre meet is then a $\QQ$-rational point, so that 
$X(\QQ)\neq \emptyset$ for the surfaces under consideration.
Thus the main arithmetic questions of interest concern whether or not 
$X(\QQ)$ is dense in $X$ under the Zariski topology, or in 
$X(\A)$ under the product topology, 
where $\A$ denotes the set of ad\`eles for $\QQ$.
Our first result resolves these basic questions.

\begin{theorem}\label{t:1}
Let $X/\PP_\QQ^1$ be a conic bundle surface 
over $\QQ$, in which degenerate 
fibres exist and are all defined over $\QQ$. Then 
the set $X(\QQ)$ is Zariski dense in $X$. Furthermore, 
the Brauer--Manin obstruction is the only obstruction to 
weak approximation for $X$.
\end{theorem}

We assume without loss of generality that $X\to\PP^1_\QQ$
is relatively minimal and the fibre at infinity is smooth.
Let $P\in \QQ[t]$ be the separable monic 
polynomial of degree $r$ that vanishes at the points of 
$\AA_\QQ^1=\PP_\QQ^1\setminus \{\infty\}$ which produce degenerate 
fibres. Our hypotheses are therefore equivalent to a 
factorisation 
$P(t)=(t-e_1)\cdots (t-e_r)$, with 
 $e_1,\ldots,e_r\in \QQ$ pairwise distinct, and  
 $a_1,\ldots,a_r\in \QQ^*\setminus {\QQ^*}^2$ such that 
each irreducible component of the fibre above $e_i$ is defined over 
$\QQ(\sqrt{a_i})$, for $1\leq i\leq r$. 

The elements of the Brauer group 
$\Br(X)=H^2_{\mathrm{\acute{e}t}}(X,\mathbb{G}_m)$ 
have the following explicit description.
Since $X(\QQ)\not=\emptyset$ the natural map $\Br(\QQ)\to \Br(X)$
is injective. Let
$$\delta:(\ZZ/2\ZZ)^r \rightarrow \QQ^*/{\QQ^*}^2$$ be the map
that sends $(n_1,\ldots,n_r)\in (\ZZ/2\ZZ)^r$ to the class of 
$\prod_{i=1}^r a_i^{n_i}$ in $\QQ^*/\QQ^{*2}$. 
By the Faddeev reciprocity law we have $a_1\cdots a_r\in{\QQ^*}^2$,
so that $(1,\ldots,1)\in{\rm Ker}(\delta)$.
For $1\leq i \leq r$, the quaternion algebras $(a_i,t-e_i)$ 
form classes in $\Br(\QQ(t))$. An integral linear combination 
$\sum_{i=1}^r n_i(a_i,t-e_i)$ gives rise to an element of $\Br(X)$
if and only if $(n_1,\ldots,n_r)\in {\rm Ker}(\delta)$. 
This defines a homomorphism ${\rm Ker}(\delta)\to\Br(X)/\Br(\QQ)$.
It is well known that it is surjective with the kernel 
generated by $(1,\ldots,1)$
(cf.\ \cite[Prop.~7.1.2]{skoro}).
The second part of Theorem~\ref{t:1} states that 
$X(\QQ)$ is dense in $X(\A)^{\Br}$, under the product of $v$-adic 
topologies, where $X(\A)^{\Br}$ denotes the left kernel in the 
Brauer--Manin pairing $X(\A)\times \Br(X)\rightarrow \ZZ/2\ZZ$.
We have $\Br(X)=\Br(\QQ)$ if and only if
${\rm Ker}(\delta)$ is generated by $(1,\ldots,1)$,
in which case $X$ satisfies weak approximation.

An important feature of Theorem \ref{t:1} is that it covers, 
unconditionally, conic bundle surfaces with arbitrarily many 
degenerate fibres. It
can be applied, for example, to the surfaces given by the equation
$$f(t)x^2+g(t)y^2+h(t)z^2=0,$$
where $t$ is a coordinate function on $\AA_\QQ^1$, 
$(x:y:z)$ are homogeneous coordinates
in $\PP^2_\QQ$, and $f,\,g,\,h$ are products of linear polynomials
with rational coefficients. 
In \S \ref{s:DP} we use Theorem~\ref{t:1} to construct explicit 
families of minimal del Pezzo surfaces $X$ of degree 1 and 2 over $\QQ$,
for which the set $X(\QQ)$ is non-empty and dense in $X(\A)^{\Br}$.

\medskip

The proof of Theorem \ref{t:1} rests upon a confirmation of 
the Hasse principle and weak approximation for 
an auxiliary class of varieties.  For the moment, let 
$a_1,\dots, a_r\in \QQ^*\setminus {\QQ^*}^2$  and let 
$f_1,\ldots,f_r\in \QQ[u_1,\ldots,u_s]$ be a system of pairwise 
non-proportional homogeneous linear polynomials,  with $s\geq 2$.
We consider smooth varieties 
$\V\subset \AA_\QQ^{2r+s}$, defined by 
\begin{equation}\label{eq:torsor}
0\neq x_i^2-a_iy_i^2 = f_i(u_1,\ldots,u_s), \quad  i=1,\ldots,r.
\end{equation}
We will establish the following result in \S \ref{s:AC}. 

\begin{theorem}\label{t:ut}
$\V(\QQ)$ is Zariski dense in $\V$ as soon as it is non-empty.
Furthermore, $\V$ satisfies the Hasse principle and 
weak approximation. 
\end{theorem}

The proof of Theorem \ref{t:ut} relies on recent work of Matthiesen \cite{lm,lm'}, which itself builds on fundamental work of Green and Tao \cite{GT} and Green--Tao--Ziegler \cite{GTZ}. Our approach is based on counting suitably constrained integer points on $\V$ and the shape of our asymptotic formula allows us to deduce the existence of a global solution close to any finite set of local solutions, provided that the system of equations is everywhere locally soluble.  This follows the pattern of the Hardy--Littlewood circle method, although it must be stressed that Theorem \ref{t:ut} is beyond the reach of the circle method.

The fact that Theorem \ref{t:1} follows from
Theorem \ref{t:ut} has been known for a long time 
(see \cite{ct-s-86}, \cite{ct-s-87}, \cite{s-conic}). We recall this argument
here, although Theorem \ref{t:1} is also a special  
case of Theorem \ref{t:1.1} below.
Let $X/\PP_\QQ^1$ be a conic bundle surface as in 
Theorem~\ref{t:1}. According to work of 
Colliot-Th\'el\`ene and Sansuc \cite[Thm.~2.6.4(iii)]{ct-s-87}, 
any universal torsor $\mathcal{T}$ over $X$ is $\QQ$-birationally
equivalent to 
$W_{\bla}\times C\times \AA_\QQ^1$, where $C$ is a conic over 
$\QQ$ and $W_{\bla}\subset \AA_\QQ^{2r+2}$ is the variety
defined by 
$$u-e_i v =\lambda_i (x_i^2-a_iy_i^2), \quad i=1,\ldots,r,$$
for suitable $\bla=(\la_1,\ldots,\la_r) \in ({\QQ^*})^{r}$.
An  application of 
Theorem \ref{t:ut} in the special case $s=2$ shows that all universal 
torsors $\mathcal{T}$ over $X$ satisfy the Hasse principle and weak approximation.
Since $X(\QQ)\neq \emptyset$ it therefore follows from the descent theory
of Colliot-Th\'el\`ene and Sansuc
\cite[Thm.~3.5.1, Prop.~3.8.7]{ct-s-87} that
$X(\QQ)$ is dense in $X(\A)^{\Br}$ under the product topology, 
as required for the second part of Theorem~\ref{t:1}.
This implies that {\em weak-weak approximation} holds, namely,
there is a finite set $S$ of places of $\QQ$ such that weak 
approximation holds away from $S$.
In particular, for almost all primes $p$ the set $X(\QQ)$ 
is dense in $X(\QQ_p)$ under the $p$-adic topology. 
This shows that the first part of Theorem \ref{t:1} follows from 
its second part.

\medskip

Recall that if $X\to B$ and $Y\to B$ are two varieties
with morphisms to the same base variety $B$, then
the fibred product $X\times_B Y$ can be defined as
the restriction of $X\times Y\to B\times B$ to the diagonal
$B\subset B\times B$. In \S \ref{descent} we shall employ 
`open descent', as in \cite{CTSk}, to prove the  following
result. 

\begin{theorem}\label{t:1.1}
Let $X_i/\PP_\QQ^1$, $i=1,\ldots,n$, be conic bundle surfaces
over $\QQ$, in which the degenerate fibres are all defined over $\QQ$. 
Let 
$$X=X_1\times_{\PP^1_\QQ}X_2\times_{\PP^1_\QQ}\cdots\times_{\PP^1_\QQ}X_n$$
be the fibred product. Assume that whenever two or more
of these conic bundles have
degenerate fibres over the same point of $\PP_\QQ^1$, the
components of their fibres at this point
are defined over the same quadratic field.
Then the Brauer--Manin obstruction is the only obstruction to 
the Hasse principle and weak approximation on any smooth and proper 
model of  $X$.
\end{theorem}

Let $a_i\in \QQ^*\setminus \QQ^{*2}$ and $c_i\in\QQ^*$ for $i=1,\ldots,n$.
Given pairwise distinct rational numbers $e_1,\ldots,e_{2n}$,
Theorem \ref{t:1.1} can be applied to the 
intersection of quadrics 
$$
(u-e_{2i-1}v)(u-e_{2i}v)=c_i(x_i^2-a_iy_i^2), \quad i=1,\ldots,n,
$$
in $\PP_\QQ^{2n+1}$.
Indeed no two of the conic bundles in the fibred product have degenerate fibres over the same point of $\PP_\QQ^1$.
The interest here is that such varieties do not
necessarily have $\QQ$-points. In fact, counter-examples to
the Hasse principle and weak approximation are known (see \cite[\S 7]{ct-c-s}).  
Theorem \ref{t:1.1} tells us that all such counter-examples are explained by the Brauer--Manin obstruction. This was previously known only when $n=2$, by using 
a descent argument to reduce the problem to an  intersection of two quadrics in $\PP_\QQ^6$ covered by   \cite[Thm.~6.7]{CT}.

\medskip

It possible to generalise Theorem \ref{t:1} to families of 
higher-dimensional quadrics. By \cite[Prop.~3.9]{CT} any variety
with a surjective morphism to an open subset of affine space, 
such that the fibres are smooth projective 
quadrics of dimension at least 3, satisfies the Hasse
principle and weak approximation. Thus  we focus
on the case of a variety with a surjective map
to $\PP^1_\QQ$ such that the fibres are 2-dimensional quadrics.
Progress so far has been restricted to the case in which there 
are at most three geometric fibres which are quadrics of rank 
2 or less, as in \cite{CTSk} and \cite{Sk}. The following 
result will be proved in \S \ref{descent}.

\begin{theorem}\label{t:1.2}
Let $X$ be a smooth, proper and geometrically integral threefold
over $\QQ$ 
equipped with a surjective morphism $X\to\PP^1_\QQ$ such that the 
generic fibre is a $2$-dimensional quadric. 
If all the fibres that are not geometrically integral
are defined over $\QQ$, then
the Brauer--Manin obstruction is the only obstruction to 
weak approximation on $X$.
\end{theorem}

\medskip

One can also deduce analogous statements 
for suitable higher-dimensional varieties. Let $m\geq 1$ and $n\geq 3$.
The equation
\begin{equation}
\sum_{i=1}^n f_i(\t) X_i^2=0,
\label{e3}
\end{equation}
defines a variety in  $\PP_\QQ^{n-1}\times\AA_\QQ^m$, where
$\t=(t_1,\ldots,t_m)$ and $f_1,\ldots,f_n\in \QQ[\t]$.  
The following result will be established in 
\S \ref{s:FM} using the fibration method.  

\begin{theorem}\label{fib}
The Brauer--Manin obstruction is the only obstruction to weak 
approximation on smooth and proper models of the varieties \eqref{e3},
provided that $f_1, \ldots,f_n$ are products of non-zero
linear polynomials defined over $\QQ$.
\end{theorem}

As mentioned above, in the case $n\geq 5$
the smooth Hasse principle and weak approximation are known to hold
for \eqref{e3} without any assumption on $f_1,\ldots,f_n$
(see \cite[Prop.~3.9]{CT}).

\begin{ack}
While working on this paper the first two authors were 
supported by EPSRC grant  \texttt{EP/E053262/1}, 
and the first author was further supported by ERC grant \texttt{306457}.
The third author was supported by the 
{\em Centre Interfacultaire Bernoulli} of the Ecole Polytechnique
F\'ed\'erale de Lausanne, whose hospitality he is grateful for. 

The authors are grateful to Jean-Louis Colliot-Th\'el\`ene
for detailed comments which have helped to 
improve the exposition of our argument. They would also like to
thank him for extremely useful discussions 
leading to a simplification of 
the proof of Theorem \ref{t:1.2} in the first version 
of this paper (available as \texttt{arXiv:1209.0207v1}).

\end{ack}

\section{Arithmetic of $\V$}
\label{s:AC}

In order to establish  Theorem \ref{t:ut} it will  suffice to  
assume that 
the varieties $\V$ in \eqref{eq:torsor} 
are everywhere locally soluble and to show, under this hypothesis, that 
$\V(\QQ)$ is non-empty and that
$\V$ satisfies weak approximation.
Since individual conics with $\QQ$-points satisfy weak approximation, it will suffice to place
weak approximation conditions on the variables $\u=(u_1,\dots,u_s)$ in $\V$
alone.

After a change of variables we may henceforth  assume without loss of generality that $a_1,\ldots,a_r$ are integers in \eqref{eq:torsor} and that 
$f_1,\dots,f_r$ are all defined over $\ZZ$.  
Let us write 
$
\{1,\ldots,r\}=I_- \cup I_+,
$
where $i\in I_{\pm}$ if and only if $\sign(a_i)=\pm$.

Let $\Omega$ denote the set of places of $\QQ$.  
For any $v\in \Omega$, 
let  $|\cdot|_v$ denote the $v$-adic norm.
When $v=\infty$ we will simply write
$|\cdot|_\infty=|\cdot|.$   
Let $S\subset \Omega$ be any finite set and let $\ve>0$. 
Without loss of generality we may suppose that $S$ contains the infinite place
and all finite places  $v$ bounded by $L$, for some parameter $L$ to be specified in due course.
We are given points 
$(\x^{(v)},\y^{(v)}, \u^{(v)})\in \V(\QQ_v)$ for every 
$v\in\Omega$.
Our task is to show that there exists a point 
$(\x,\y, \u)\in \V(\QQ)$ such that 
\begin{equation}\label{eq:local-uv}
|\u-\u^{(v)}|_v<\ve,
\end{equation}
for each  $v\in S$.

On rescaling appropriately it suffices to assume that  the solutions 
$(\x^{(v)},\y^{(v)}, \u^{(v)})$ we are given belong
to $\ZZ_v^{2r+s}$ for every finite $v\in S$.
Applying the Chinese remainder theorem 
for $\ZZ^{2r+s}$ we can then produce an
integer vector
$(\x^{(M)},\y^{(M)}, \u^{(M)})$ such that 
\begin{align*}
 |\x^{(M)}-\x^{(v)}|_v<\ve, \quad
 |\y^{(M)}-\y^{(v)}|_v<\ve, \quad
 |\u^{(M)}-\u^{(v)}|_v<\ve, 
\end{align*}
for all finite $v \in S$.
We will seek $(\x,\y,\u)\in \mathcal{V}(\ZZ)$ satisfying the necessary local conditions.
For the finite places \eqref{eq:local-uv} becomes
\begin{equation}\label{eq:local-cong} 
u_j\equiv u^{(M)}_j \bmod{M}, \quad j=1,\ldots,s,
\end{equation}
for an appropriate modulus $M\in\NN$.
Suppose $\val_p(M)=m$ for $m\in \NN$. 
For technical reasons we require that 
\begin{equation}\label{eq:technical'}
m \geq  \max_{1\leq i\leq r}\{\val_p(4a_i)\}
\end{equation}
and 
\begin{equation}\label{eq:technical}
 f_i(\u^{(M)}) \not\= 0 \bmod{p^{m}}, \quad i=1,\ldots,r.
\end{equation}
Since $f_i(\u^{(v)}) \neq 0$ in $\QQ_v$ we can arrange for these
properties to hold providing that we take $\ve$ to be  sufficiently small.

For the infinite place we will seek points satisfying 
\begin{equation}\label{eq:local-inf}
 |\u-B\u^{(\infty)}|<\ve B,
 \end{equation}
with $B=C^2$ for a positive integer $C\equiv 1 \bmod{M}$ tending to infinity.
We have  $f_i(\u^{(\infty)})>0$ for $i\in I_-$,
since
$(\x^{(\infty)},\y^{(\infty)}, \u^{(\infty)})$ belongs to $\V(\RR)$. 
For $\ve$ sufficiently small, therefore,  any $\u \in \RR^s$ satisfying
\eqref{eq:local-inf} will  produce positive values of $f_i(\u)$ for
$i\in I_-$.
Finally, it is clear that any solution 
$(\x,\y,\u)\in \V(\ZZ)$
satisfying \eqref{eq:local-cong} and \eqref{eq:local-inf} will give rise to a solution
$(C^{-1}\x,C^{-1}\y,C^{-2}\u)\in \V(\QQ)$
satisfying our original condition \eqref{eq:local-uv}.

Let  $q_i(x,y)=x^2-a_iy^2$, for $i=1,\ldots,r$.
This is a primitive binary quadratic form of discriminant $4a_i$, which is positive definite for $i\in I_-$ and indefinite for $i\in I_+$.
For $d\leq -4$ let
$$
w(d)=\begin{cases}
4, &\mbox{if $d=-4$,}\\
2, &\mbox{if $d<-4$,}
\end{cases}
$$
and for $d>0$ let $\eta(d)$ denote the fundamental unit of $\QQ(\sqrt{d})$.
Let us call a solution 
$(\x,\y,\u)\in \ZZ^{{2r+s}}$ of \eqref{eq:torsor} {\em primary} if the 
pair $(y_i,z_i)$ lies in a fundamental domain for the action of the group of automorphs $\mathcal{E}_i$ of $q_i$, for $i=1,\ldots, r$.
Our strategy will be to estimate asymptotically, as $B\rightarrow \infty$, the total number $N(B)$ of primary solutions
$(\x,\y,\u)\in \ZZ^{{2r+s}}$
of  \eqref{eq:torsor}, which satisfy \eqref{eq:local-cong} and \eqref{eq:local-inf}.

We will henceforth view $\ve, M$, together with the coefficients of $\V$ and 
$\u^{(M)},\u^{(\infty)}$ as being fixed once and for all.  Any implied constants in this section 
will therefore be allowed to depend on these quantities.
Given  $n\in \ZZ$, we define the representation functions
$$
R_i(n)=\#\{(x,y)\in \ZZ^2/\mathcal{E}_i: q_i(x,y)=n\}, 
$$ 
for $i=1,\ldots,r$,  with  $R_i(n)=0$ if $n\leq 0$ and $i\in I_-$.
We may clearly write
$$
N(B)=
\sum_{\substack{
\u \in \ZZ^s\\
\mbox{
\scriptsize{
 \eqref{eq:local-cong} and \eqref{eq:local-inf} hold}}}}
\prod_{i=1}^r
R_i(f_i(\u)).
$$
The success of our investigation rests upon work of the 
 second author
\cite{lm,lm'}, who has shown how 
recent innovations 
of Green and Tao \cite{GT} and Green--Tao--Ziegler \cite{GTZ} 
in additive combinatorics can 
be brought to bear on sums like $N(B)$ for arbitrary $r$.

We eliminate the constraint \eqref{eq:local-cong} in $N(B)$
by writing 
$u_j=u_j^{(M)}+Mt_j$ for $j=1,\ldots,s$. This leads to the expression
$$
N(B)=
\sum_{\substack{\t\in \ZZ^s\cap K}}
\prod_{i=1}^r
R_i(g_i(\t)),
$$
where $g_i(\t)=
f_i(\u^{(M)}+M\t)$ and 
$$
K=
\left\{\t\in \RR^s: 
 \begin{array}{l}
 |M\t+\u^{(M)}-B\u^{(\infty)}|<\ve B 
 \end{array}
\right\}.
$$
This region is convex and 
contained in $[-cB,cB]^s$ 
for an appropriate absolute positive constant $c$.
It has measure $(2\ve M^{-1}B)^s\gg B^s$.
Our choice of $\ve$ ensures that  
$g_i(K)$ is positive for every $i \in I_-$.
Moreover,
 $(g_1,\ldots,g_r):\ZZ^s\rightarrow \ZZ^r$ 
defines a system of linear polynomials of `finite complexity', in the language of Green and Tao \cite{GT}.
Given $A\in \ZZ$ and $q\in \NN$, let
$$
\rho_i(q;A)=\#\{ (x,y)\in (\ZZ/q\ZZ)^2: x^2-a_iy^2\equiv A \bmod{q}\}.
$$
It therefore follows from   \cite[Thm.~1.1]{lm'} that
\begin{equation}\label{eq:final}
N(B)=\beta_\infty \prod_p \beta_p +o(B^s), 
\end{equation}
as $B\rightarrow \infty$. Here the main term is a product of local densities, given by
$$
\beta_\infty = 
\meas(K)\prod_{i\in I_-}
\frac{\pi }{w(4a_i)\sqrt{|a_i|}} \prod_{j\in I_+} \frac{\log \eta(a_j)}{\sqrt{a_j} }
$$
and 
$$
\beta_p=\lim_{k\rightarrow \infty} p^{-(s+r)k} G(p^k),
$$
for each prime $p$,
where
\begin{align*}
G(p^k)&=
\sum_{\substack{\t\in (\ZZ/p^k\ZZ)^s }}
\prod_{i=1}^r
\rho_i(p^k;g_i(\t))\\
&=
\#\left\{
(\x,\y,\t)\in (\ZZ/p^k\ZZ)^{2r+s}:  
\begin{array}{l}
x_i^2-a_iy_i^2\equiv g_i(\t) \bmod{p^k}\\
\mbox{for $i=1,\ldots,r$}
\end{array}
\right\}.
\end{align*}
Since  $\beta_\infty\gg \meas(K)\gg  B^s$, we see that in order to complete the proof of  Theorem \ref{t:ut} it remains to show that 
$\prod_p \beta_p \gg1 $
in \eqref{eq:final}.

For each prime $p$, let
$$
\beta'_p = 
\lim_{k\rightarrow \infty} p^{-(s+r)k}
\sum_{\substack{\u \in (\ZZ/p^k\ZZ)^s }}
\prod_{i=1}^r
\rho_i(p^k;f_i(\u)) 
$$
be the local factor associated to the original system of equations.
By \cite[Lemma~8.3]{lm} these factors satisfy $\beta'_p = 1 + O(p^{-2})$. 
We may now specify the parameter $L=O(1)$ to be such that $\beta'_p > 0$ for all $p>L$.
For primes $p \nmid M$ this choice of $L$ ensures that
$$
\prod_{p\nmid M}\beta_p = \prod_{p\nmid M}\beta'_p \gg 1,
$$
since the change of variables is non-singular in this case.
Our final task is to show that $\beta_p > 0$ for primes $p\mid M$. 
Suppose $\val_p(M)=m>0$. 
We have arranged things so that 
\eqref{eq:technical'} and \eqref{eq:technical} hold.
The integer vector $(\x^{(M)},\y^{(M)}, \u^{(M)})$ satisfies \eqref{eq:torsor} modulo $M$, which implies $G(p^{m}) \geq p^{sm}$.
To analyse $G(p^{k})$ for $k>m$, we shall employ
\cite[Cor.~6.4]{lm}, which yields
$$
\rho_i(p^k,A) = \frac{1}{p} \rho_i(p^{k+1},A+\ell p^{k})
$$
for any $\ell \in\ZZ/p\ZZ$, providing that $k\geq \val_p(4a_i)$ and $A\not\=0\bmod{p^{k}}$.
Since these conditions hold when $A=g_i(\t)$ and  $\t \in \ZZ^s$, for $k>m$,
we deduce that
$
G(p^{k+1})= p^{s+r}G(p^{k}).
$
Hence
$$
\beta_p 
= p^{-(s+r)m} G(p^{m}) 
\geq p^{-rm}>0
$$
for $p\mid M$. This establishes the desired lower bound for the product 
of local densities and so concludes the proof of Theorem \ref{t:ut}.

\section{Descent} \label{descent}

Let $Y$ be a variety over a number field $k$, and let
$f:Z\to Y$ be a torsor of a $k$-torus $T$. We write $\A_k$
for the ring of ad\`eles of $k$. Specialising the torsor at an adelic
point defines the evaluation map $Y(\A_k)\to \prod_v H^1(k_v,T)$,
where the product is taken over all completions $k_v$ of $k$.
Let $Y(\A_k)^f$ be the set of adelic points
for which the image of the evaluation map is contained in the
image of the natural map $H^1(k,T)\to \prod_v H^1(k_v,T)$.
It is clear that the diagonal image of $Y(k)$ in $Y(\A_k)$
is in $Y(\A_k)^f$.

There is an equivalent way to define $Y(\A_k)^f$.
Up to isomorphism, 
the $k$-torsors $R$ of $T$ are classified by their classes $[R]\in
H^1(k,T)$. The twist of $f:Z\to Y$
by $R$ is defined as the quotient of $Z\times R$ by the diagonal
action of $T$, with the morphism to $Y$ induced by the first projection.
We denote the twisted torsor by $f^R:Z^R\to Y$.
Then $Y(\A_k)^f$ is the union of the images of 
projections $f^R:Z^R(\A_k)\to Y(\A_k)$, 
for all $[R]\in H^1(k,T)$ (see \cite[\S 5.3]{skoro}).

The following slight variation on \cite[Prop.~1.1]{CTSk} was found
after our discussions with J.-L.~Colliot-Th\'el\`ene.

\begin{proposition} \label{p1}
Let $X$ be a smooth geometrically integral variety
over a number field $k$. 
Let $Y\subset X$ be a dense open set, and let
$f:Z\to Y$ be a torsor of a $k$-torus $T$. 
Then $X(\A_k)^\Br\not=\emptyset$ implies $Y(\A_k)^f\not=\emptyset$.
If $X$ is proper, then $X(\A_k)^\Br$ is contained in the
closure of $Y(\A_k)^f$ in $X(\A_k)=\prod_v X(k_v)$.
In this case, if all the twists of $Y$ by $k$-torsors of $T$
satisfy the Hasse principle and weak approximation,
then $X(k)$ is dense in $X(\A_k)^\Br$.
\end{proposition}
\begin{proof}
Let $\hat T$ be the group of homomorphisms of algebraic groups 
$T\times_k\bar k\to \mathbb{G}_{m,\bar k}$. Equipped with discrete 
topology, $\hat T$ is a continuous ${\rm Gal}(\bar k/k)$-module.
The natural pairing of discrete ${\rm Gal}(\bar k/k)$-modules
$T(\bar k)\times \hat T \to \bar k^*$
gives rise to the $\cup$-product pairing
$$\cup: H^1_{\rm\acute et}(Y,T)\times H^1(k,\hat T) \to 
H^1_{\rm\acute et}(Y,T)\times H^1_{\rm\acute et}(Y,\hat T) \to 
H^2_{\rm\acute et}(Y,\mathbb{G}_m)=\Br(Y),$$
cf. \cite[pp.~63-64]{skoro}. 
Let $[Z]\in H^1_{\rm\acute et}(Y,T)$ be
the class of the torsor $Z/Y$, and let
$B\subset \Br(Y)$ be the subgroup $[Z]\cup H^1(k,\hat T)$.
Since $H^1(k,\hat T)$ is finite, $B$ is also finite.
Let $Y(\A_k)^B$ be the set of adelic points of $Y$ that are
orthogonal to $B$ with respect to the Brauer--Manin pairing.
By \cite[Prop. 1.1]{CTSk} (based on Harari's `formal lemma'
\cite[Cor. 2.6.1]{harari}) 
we have $X(\A_k)^{B\cap\Br(X)}\not=\emptyset$
if and only if $Y(\A_k)^B\not=\emptyset$, and 
the latter set is dense in the former when $X$ is proper. 
Since $X(\A_k)^\Br\subset X(\A_k)^{B\cap\Br(X)}$,
it remains to prove that $Y(\A_k)^B=Y(\A_k)^f$. 
This is a well known consequence of the Poitou--Tate duality for tori,
see, e.g., the proof of statement (2) in \cite[pp.~115, 119--121]{skoro}.
\end{proof}

\begin{proof}[Proof of Theorem \ref{t:1.1}]
Without loss of generality we assume that $X_j\to\PP^1_\QQ$
is relatively minimal and the fibre at infinity
of $X_j\to\PP^1_\QQ$ is smooth, for each $j=1,\ldots,n$.
Then there are $e_1,\ldots, e_r$ in $\QQ=\AA^1_\QQ(\QQ)$
such that the restriction of $X_j\to\PP^1_\QQ$ to $\PP^1_\QQ
\setminus\{e_1,\ldots, e_r\}$ is a smooth morphism, for $j=1,\ldots,n$.
By assumption, for $i=1,\ldots,r$ there exists 
$a_i\in\QQ^*\setminus\QQ^{*2}$, 
defined up to a square, such that the fibre of each $X_j\to\PP^1_\QQ$
at $e_i$ is either a smooth conic or a union of two conjugate
lines defined over $\QQ(\sqrt{a_i})$. 

Let $U=\AA^1_\QQ\setminus\{e_1,\ldots,e_r\}$. For $j=1,\ldots,n$
define $Y_j\subset X_j$ as the inverse image of $U\subset\PP^1_\QQ$, 
and let $Y$ be the fibred product of $Y_1,\ldots,Y_n$ over $U$.
To apply Proposition \ref{p1} we now introduce a certain torsor
over $Y$.
Let $\W_\bla\subset\AA^{2r+2}_\QQ$, for $\bla\in(\QQ^*)^{r}$, be the variety given by 
\begin{equation}
u-e_i v =\lambda_i (x_i^2-a_iy_i^2), \quad i=1,\ldots,r,
\quad \quad v\prod_{i=1}^r(u-e_iv)\not=0. \label{W}
\end{equation} 
The morphism $\W_\bla\to U$ that sends the point $(u,v,x_i,y_i)$ 
to the point with the coordinate $t=u/v$, is a torsor of
the following $\QQ$-torus $T$:
$$v=x_1^2-a_1y_1^2=\cdots=x_r^2-a_r y_r^2\not=0.$$
The fibred product $Y\times_{U}\W_\bla$ is a $Y$-torsor of $T$,
for any $\bla$. 

The $\QQ$-torsors of $T$ are the affine varieties $R_\c$ given by
$$v=c_1(x_1^2-a_1y_1^2)=\cdots=c_r(x_r^2-a_r y_r^2)\not=0,$$
where $\c=(c_1,\ldots,c_r)\in (\QQ^*)^r$.
The isomorphism classes of $\QQ$-torsors of $T$ bijectively correspond
to $\c\in (\QQ^*)^r$ up to
a common non-zero rational multiple and multiplication of each
$c_i$ by the norm of a non-zero element of $\QQ(\sqrt{a_i})$. 
The twist $\W_\bla^{R_\c}$ is the torsor  $\W_{\c\bla}$,
where $\c\bla=(c_1\la_1,\ldots,c_r\la_r)$.
Thus the set of torsors $Y\times_{U}\W_\bla \to Y$ for all $\bla\in(\QQ^*)^{r}$
is closed under all twists by $\QQ$-torsors of $T$.

For $j=1,\ldots,n$ we denote by $I_j$ the subset of $\{1,\ldots,r\}$ 
such that the fibre of $X_j\to\PP^1_\QQ$ at $e_i$ is singular if and 
only if $i\in I_j$. Let $r_j=|I_j|$. We define 
$\W_{\bla}^{(j)}\subset \AA_\QQ^{2r_j+2}$ to be the variety given by
$$u-e_i v =\lambda_i (x_i^2-a_iy_i^2), \quad i\in I_j,
\quad v\prod_{i=1}^r(u-e_iv)\not=0.$$ 
for $\bla\in(\QQ^*)^{r_j}$.
As proved in \cite{ct-s-87} (Thm.~2.6.4(ii)(a) and
Remarque~2.6.8), there exists a conic $C_j$ over $\QQ$ such that 
$Y_j\times_{U}\W_\bla^{(j)}$ is birationally equivalent
to $C_j\times \W_\bla^{(j)}$. There is a natural morphism
$\W_\bla\to \W_\bla^{(j)}$ that forgets the coordinates $x_i,\,y_i$
for $i\notin I_j$.
This morphism is obviously compatible with the projection to $U$,
hence $Y_j\times_{U}\W_\bla$ is birationally equivalent
to $C_j\times \W_\bla$. Therefore $Y\times_{U} \W_\bla$
is birationally equivalent to $C_1\times\cdots \times C_n\times \W_\bla$.
By Theorem~\ref{t:ut} and the Hasse--Minkowski theorem this variety
satisfies the Hasse principle and weak approximation. 
It now follows from Proposition \ref{p1} that $X(\QQ)$ is dense
in $X(\A_\QQ)^\Br$.
\end{proof}

\medskip

Let us now turn to the arithmetic of pencils of 2-dimensional quadrics.
We start with recalling the relevant definitions from \cite{Sk}. 
Let $k$ be an arbitrary field of odd characteristic.

\begin{definition} \label{d1}
(1) A geometrically integral variety $X$ over $k$ with a morphism
$p:X\to\PP^1_k$ is a {\em quadric bundle} if
every closed point $P\in \PP^1_k$ 
has a Zariski open neighbourhood $U_P\subset \PP^1_k$
such that $p^{-1}(U_P)$ is the closed subset of $U_P\times \PP^3_k$
given by the vanishing of a quadratic form $Q_P(x_1,x_2,x_3,x_4)=0$
with coefficients in the $k$-algebra of regular functions on $U_P$, 
such that $\det(Q_P)$ is not identically zero.

(2) A quadric bundle $X/\PP^1_k$ is {\em admissible} if for
every closed point $P\in \PP^1_k$ for which the fibre 
$X_P$ is singular, $U_P$
and $Q_P$ in (1) can be chosen so that $Q_P(x_1,x_2,x_3,x_4)=
\sum_{i=1}^4 f_i x_i^2$, where $f_i$ is 
invertible outside $P$ with at most a simple zero at $P$, 
and $f_1(P)f_2(P)\not=0$.

(3) Let us call an admissible quadric bundle {\em relatively minimal}
if, in the notation of (2), for each closed point $P\in \PP^1_k$ 
such that $f_3(P)=f_4(P)=0$ the (well-defined) values of the
functions $-f_1/f_2$ and $-f_3/f_4$ at $P$ are both non-squares
in the residue field $k(P)$.
\end{definition}

If $X/\PP^1_k$ is a relatively minimal admissible quadric bundle,
then the closed fibre $X_P$ is not geometrically integral if and only if
$X_P$ is the zero set of a quadratic form of rank 2.
In our notation, $X_P$ is given by $f_1(P)x_1^2+f_2(P)x_2^2=0$.
Thus $X_P$ is the union of two conjugate projective planes defined
over the quadratic extension $k(P)(\sqrt{a_P})$ of the residue field $k(P)$,
where $a_P=-f_1(P)/f_2(P)$. In particular, the (non-trivial)
class of $a_P$ in $k(P)^*/k(P)^{*2}$ is uniquely determined by $X/\PP^1_k$.

The singular locus $(X_P)_{\rm sing}$ of $X_P$ is the projective
line given by $x_1=x_2=0$.
An easy calculation (see \cite[Cor. 2.1]{Sk}) shows that 
the singular locus $X_{\rm sing}$ is contained in the union of singular loci
of the closed fibres of $X/\PP^1_k$ that are not geometrically integral.
Let $b_P\in k(P)^*$ be the value of $-f_3/f_4$ at $P$. By \cite[Prop.~2.2]{Sk},
$X_{\rm sing}\cap X_P$ is the 
subscheme of $(X_P)_{\rm sing}$ given by $x_4^2=b_Px_3^2$.
In particular, the (non-trivial) class of $b_P$ in
$k(P)^*/k(P)^{*2}$ is uniquely determined by $X/\PP^1_k$.

Recall that a scheme over $k$ is called {\em split} if it contains 
a non-empty geometrically integral open subscheme
\cite[Def.~0.1, p. 906]{S96}.
Let us denote by $\tilde X$ the blow-up of $X_{\rm sing}$ in $X$.
In \cite[Prop.~2.4]{Sk} it is shown that $\tilde X$ is a 
smooth projective threefold.
Since $X/\PP^1_k$ is relatively minimal, each fibre 
of $\tilde X/\PP^1_k$ that is not geometrically integral consists
of two irreducible components, none of them geometrically
integral (since $a_P$ and $b_P$ are both non-squares in
$k(P)^*$), cf. \cite[Remark 2.2]{Sk}. Hence a fibre
of $\tilde X/\PP^1_k$ is split if and only if it is geometrically integral.

\begin{proof}[Proof of Theorem \ref{t:1.2}] 
By \cite{Sk} (Prop.~2.1 and its proof, Prop.~2.3) 
there exists a relatively minimal admissible quadric bundle 
$X'/\PP^1_\QQ$ such that the generic fibres of $X/\PP^1_\QQ$
and $X'/\PP^1_\QQ$ are isomorphic. (In particular, 
$X$ and $X'$ are birationally equivalent.) 
If a fibre $X_P$ is geometrically integral, hence split, 
then $\tilde X'_P$ is split too \cite[Cor.~1.2]{S96}.
By the previous paragraph $\tilde X'_P$ 
is then a geometrically integral quadric, hence so is $X'_P$.
It follows that $X'_P$ is geometrically integral whenever $X_P$ is 
geometrically integral.

If all the fibres of $X'/\PP^1_\QQ$ are geometrically integral,
the variety $X$ satisfies the Hasse principle and weak approximation
(see \cite[Thm.~3.10]{CT} or \cite[Thm.~2.1]{S96}). Thus we may assume
that at least one $\QQ$-fibre $X'_P$ of $X'/\PP^1_\QQ$ is given by a quadratic
form of rank 2. (Then almost all $\QQ$-points on the common line 
of the two planes of $X'_P$ are smooth in $X'$, hence $X(\QQ)\not=\emptyset$.)

Let us choose $\AA^1_\QQ\subset\PP^1_\QQ$ so that the fibre of 
$X'/\PP^1_\QQ$ at infinity is smooth, and let $t$ be a coordinate function
on $\AA^1_\QQ$.
By assumption we know that there are $e_1,\ldots,e_r\in\QQ$ 
such that the fibres $X'_{e_1},\ldots,X'_{e_r}$ can be given by
quadratic forms of rank 2, and all the other fibres of $X'/\PP^1_\QQ$
are geometrically integral. Let $a_1,\ldots,a_r\in\QQ^*\setminus\QQ^{*2}$,
defined up to squares,
be such that $\QQ(\sqrt{a_i})$ is the quadratic field over which
the components of $X'_{e_i}$ are defined.

Let $U=\AA^1_\QQ\setminus\{e_1,\ldots,e_r\}$, and let $U_i$ be 
a Zariski open neighbourhood of $e_i$ as in Definition \ref{d1} (2).
The restriction of $X'\to \PP^1_\QQ$ to $U_i$ can be given by
the equation
\begin{equation}
x_1^2-\alpha_i x_2^2+\gamma_i(t-e_i)(x_3^2-\beta_i x_4^2)=0,\label{e1}
\end{equation}
where $\alpha_i,\,\beta_i,\,\gamma_i$ are invertible regular functions 
on $U_i$. We then have $a_i=\alpha_i(e_i)$.

We denote by $S$ a finite set
of places of $\QQ$ containing 2 and the real place. Define $\ZZ_S$ 
as the subring of $\QQ$ consisting of the fractions
with denominators divisible only by primes in $S$.
We choose $S$ large enough so that for all $i=1,\ldots,r$ we have
$$e_i\in \ZZ_S,\quad a_i\in\ZZ_S^*,\quad
e_i-e_j\in \ZZ_S^*\quad \text{for $i\not=j$.}$$
Moreover, by further increasing $S$ we can assume that $X'$ 
has an integral model $\mathcal X'\to\PP^1_{\ZZ_S}$ 
such that for any $p\notin S$ its reduction modulo $p$
is an admissible quadric bundle
$\mathcal X'_{\FF_p}\to \PP^1_{\FF_p}$ with exactly $r$ fibres 
that are quadrics of rank 2 
at the reductions of $e_1,\ldots,e_r$ modulo $p$. For $i=1,\ldots,r$ we
define $\mathcal U_i\subset \PP^1_{\ZZ_S}$ as the complement
to the Zariski closure of $\PP^1_\QQ\setminus U_i$ in $\PP^1_{\ZZ_S}$. 
It is clear that $U_i=\mathcal U_i\times_{\ZZ_S}\QQ$.
By enlarging $S$ we ensure that
$\alpha_i,\,\beta_i,\,\gamma_i$ are invertible regular functions 
on $\mathcal U_i$, and (\ref{e1}) is an equation for 
$\mathcal X'$ over $\mathcal U_i$.

Let $a_0=a_1\cdots a_r$. 
For $\bla\in(\QQ^*)^{r}$ we define the variety $\W_\bla$ as follows:
\begin{equation}
u-e_i v =\lambda_i (x_i^2-a_iy_i^2)\not=0, \quad i=1,\ldots,r,
\quad v=x_0^2-a_0y_0^2\not=0. \label{WW}
\end{equation} 
The morphism $\W_\bla\to U$ that sends the point $(u,v,x_i,y_i)$ 
to the point with the coordinate $t=u/v$, is a torsor of
the following $\QQ$-torus $T$:
$$x_0^2-a_0y_0^2=x_1^2-a_1y_1^2=\cdots=x_r^2-a_r y_r^2\not=0.$$
Let $Y\subset X'$ be the inverse image of $U$.
The fibred product $Y\times_{U}\W_\bla$ is a $Y$-torsor of $T$,
for any $\bla$. As in the proof of Theorem \ref{t:1.1} we see that
the family of torsors $Y\times_{U}\W_\bla \to Y$ 
is closed under all twists by $\QQ$-torsors of $T$. 
By Proposition \ref{p1} it is enough to prove that
the varieties $Y\times_{U}\W_\bla$ satisfy the 
Hasse principle and weak approximation. 

Write $\W=\W_\bla$. Let us enlarge the set $S$ by including into it
the primes where we need to approximate.
We are given a family of $\QQ_p$-points $N_p$, for all primes $p$,
and a real point $N_\infty$, in $Y\times_{U}\W$.
Let $M_p$, $M_\infty$ be the images of these points in $\W$.
By Theorem \ref{t:ut} the variety $\W$ satisfies the Hasse principle
and weak approximation. Indeed, if $a_0\notin\QQ^{*2}$, then Theorem \ref{t:ut}
can be directly applied to $\W$. For $a_0\in\QQ^{*2}$ a change
of variables in the last equation of (\ref{WW})
gives $v=x'_0y'_0$, so that $\W$ is birationally
equivalent to the product of $\AA^1_\QQ$ and the variety (\ref{W}),
to which Theorem \ref{t:ut} can be applied. 

Thus in all cases we can find a point $M\in \W(\QQ)$
arbitrarily close to the points $M_\infty$ and $M_p$ for $p\in S$,
in their respective local topologies. 
Let $P\in U(\QQ)$ be the image
of $M$. We can choose $M$ so that $P$ is contained in a given non-empty
open subset of $\PP^1_\QQ$, for example in the open set 
$U_0\subset U\cap U_1\cap\cdots\cap U_r$ defined by the property
that $Y_P=X'_P$ is a smooth quadric for any $P$ in $U_0$.
Then $Y_P$ can be given by equation (\ref{e1}) for any $i=1,\ldots,r$.
By the implicit function theorem $Y_P$ has $\QQ_p$-points 
close to $N_p$ for $p\in S$ and a 
real point close to $N_\infty$. If we can prove that
$Y_P(\QQ_p)\not=\emptyset$ for all $p\notin S$, then $Y_P$
is everywhere locally soluble over $\QQ$, and hence has a $\QQ$-point
and satisfies weak approximation (by the theorem of Hasse and
the rationality of a smooth quadric with a $\QQ$-point).
This implies that $Y\times_{\PP^1_\QQ}\W$ also
has a $\QQ$-point and satisfies weak approximation.

Let $\W_0$ be the inverse image of $U_0$ in $\W$. To finish the proof
it is enough to show that the natural projection
$(Y\times_{U}\W_0)(\QQ_p)\to \W_0(\QQ_p)$
is surjective for all $p\notin S$.

We can assume that a point in $\W_0(\QQ_p)$ has coordinates
$(x_0,y_0,\ldots,x_r,y_r)\in \ZZ_p^{2r+2}$, not all divisible by $p$. 
It maps to $P=(u:v)\in U_0(\QQ_p)$, where $u,\,v\in\ZZ_p$, and
$t=u/v\in \QQ_p$ is such that $t\not=e_i$, for any $i=1,\ldots,r$.
Let us denote by $x\mapsto\bar x$ the map $\QQ_p\to\FF_p\cup\{\infty\}$
such that $\bar x\equiv x\bmod p$ if $x\in \ZZ_p$, and 
$\bar x=\infty$ if $x\in \QQ_p\setminus\ZZ_p$.
We have three possible cases:
\begin{itemize}
\item[(a)] $\bar t$ is not equal to any of the points $\bar e_i$, 
for $i=1,\ldots,r$;
\item[(b)] 
$\bar t=\bar e_i$ for some $i\in \{1,\ldots,r\}$
and $\val_p(v)$ is even;
\item[(c)] $\bar t=\bar e_i$ for some $i\in \{1,\ldots,r\}$ and 
$\val_p(v)$ is odd.
\end{itemize}
In case (a) the quadric $Y_P$ 
reduces to a geometrically integral quadric over $\FF_p$.  
Such a quadric has smooth $\FF_p$-points and any smooth $\FF_p$-point
lifts to a $\QQ_p$-point on $Y_P$ by Hensel's lemma.

Now suppose that we are in case (b) or case (c). 
Then the reduction of $Y_P$ 
modulo $p$ is the same as that of $Y_{e_i}$. 
If $a_i$ is a square modulo $p$,
the reduction of $Y_P$ modulo $p$ is a union of two projective
planes defined over $\FF_p$. Any $\FF_p$-point not on the common
line of the two planes is smooth and hence
lifts to a $\QQ_p$-point in $Y_P$ by Hensel's lemma. Now assume that
$a_i$ is not a square modulo $p$. Since $P=(t:1)\in U_i(\QQ)$
we can evaluate (\ref{e1}) at $P$ and obtain an equation for $Y_P=X'_P$.
From (\ref{WW}) we see that $\val_p(u-e_iv)$ must be even. 

In case (b) we deduce that $\val_p(t-e_i)$ is also even. 
But then $Y_P$ can be given by a quadratic form over $\ZZ_p$ that reduces
to a rank 4 quadratic form over $\FF_p$. This implies that
$Y_P$ has a $\QQ_p$-point. 

The case (c) is not compatible with the condition that $a_i$ 
is not a square modulo $p$. 
Indeed, if $\val_p(v)$ is odd, then $\val_p(t-e_i)>0$
is also odd. Take any $j\in \{1,\ldots,r\}$ with $j\not=i$. 
Since $e_i-e_j\in\ZZ_S^*$ we see that $t-e_j\in\ZZ_S^*$, so that
$u-e_j v$ has odd valuation. This implies that $a_j$
is a square modulo $p$. Since $v=x_0^2-a_0y_0^2$ has odd valuation,
$a_0$ must also be a square modulo $p$. This is a contradiction
with the fact that $a_0\cdots a_r$ is a square.
This finishes the proof of the theorem.
\end{proof}

\section{Higher-dimensional varieties}
\label{s:FM}

The purpose of this section is to establish Theorem \ref{fib}.
By \cite[Prop.~3.9]{CT} it suffices to assume that 
$n=3$ or $n=4$. On multiplying  \eqref{e3} and each of 
the variables $X_i$ by an appropriate non-zero rational
function in $\t=(t_1,\ldots,t_m)$, it suffices to replace \eqref{e3}
by a $\QQ$-birationally equivalent variety which is given by 
an equation of the same form satisfying the following additional conditions.
There exist pairwise non-proportional polynomials 
$l_1,\ldots,l_r\in \QQ[\t]$ of total degree $1$, 
such that for $j=1,\ldots,n$ we can write
$f_j=c_j\prod_{i\in I_j}l_i$ where
$c_j\in \QQ^*$ and $I_j\subseteq \{1,\ldots,r\}$. Moreover,
for $n=3$ (resp.\ $n=4$) each $l_i$ 
divides exactly one of $f_1,f_2,f_3$ (resp.\
one or two of $f_1,f_2,f_3, f_4$).
Finally, we may assume that 
$$l_i(1,0,\ldots,0)=1, \quad i=1,\ldots,r. $$
Indeed, there exists a non-zero vector $\a\in \QQ^m$ 
such that $l_i(\a)\neq 0$ for $i=1,\ldots,r$.
Assuming without loss of generality that $a_1\neq 0$, 
one achieves the claim by making the change of variables 
$t_1=a_1t_1'$ and $t_i=t_i'+a_it_1'$ for $2\leq i\leq m$,
and then replacing $c_j$ by $c_j\prod_{i\in I_j}l_i(\a)$. 
The case when \eqref{e3} is a quadric over $\QQ$ being a
subject of the Hasse--Minkowski theorem,
we can assume without loss of generality that $f_1$ is not
constant and is divisible by $l_1(\t)$.

Let us denote the variety in \eqref{e3} by $V$.
The map $p:V\to\AA_\QQ^{m-1}$ sending
$(X_1,\ldots,X_n,\t)$ to $(t_2,\ldots,t_m)$ is a surjective morphism. 
The fibre $V_\b=p^{-1}(\b)$ above 
a point $\b=(b_2,\ldots,b_m)$ of $\AA_\QQ^{m-1}$ is given by
the following equation with coefficients in the residue field $\QQ(\b)$:
$$\sum_{j=1}^n\tilde f_j(t)X_j^2=0,$$
where $\tilde f_j(t)=f_j(t,\b)$. 
We note that the morphism $p$ has a section $s$ that sends 
$(t_2,\ldots,t_m)$ to the
point of $V$ with coordinates $X_1=1$, $X_2=\ldots=X_n=0$,
$t_1=-l_1(0,t_2,\ldots,t_m)$.

Theorem \ref{fib} will follow from a variant of the fibration method
with a section, which is a result of Harari \cite[Thm. 4.3.1]{harari},
once we check that :

\begin{enumerate}\item the generic fibre $V_\eta$ of $p$ is 
geometrically integral and geometrically rational, and the section $s$
defines a smooth point of $V_\eta$;
\item there is a non-empty open subset $U\subset \AA_\QQ^{m-1}$ such that
for any $\b\in U(\QQ)$
the Brauer--Manin obstruction is the only obstruction
to weak approximation on smooth and proper models of $V_\b$.
\end{enumerate}

Let $U\subset \AA^{m-1}_\QQ$ be the open subset given
by $l_{i_1}(0,\b)\not=l_{i_2}(0,\b)$ for all $i_1\not=i_2$. 
This set is not empty since no two polynomials $l_{i_1}$ and $l_{i_2}$
are equal for $i_1\not=i_2$.
The restriction of $p$ to $U$ has geometrically integral fibres, as follows
from our assumption that if $n=3$ (resp.\ $n=4$) then each $l_i$ 
divides exactly one of $f_1,f_2,f_3$ (resp.\ one or two of 
$f_1,f_2,f_3, f_4$). 
Thus for any $\b$ in $U$ the fibre $V_\b$
is a conic bundle or an admissible quadric bundle
(see Definition \ref{d1} (2) above).
In particular, this is true for the generic fibre $V_\eta$.
The conic bundles are smooth, so for $n=3$ 
the point of $V_\eta$ defined by $s$ is certainly smooth.
In the case $n=4$ an easy calculation (cf.\ the remarks
after Definition \ref{d1}) shows that at every point of 
the singular locus 
$(V_\eta)_{\rm sing}$ exactly two of the coordinates $X_1,
X_2,X_3,X_4$ must vanish. Hence the point of $V_\eta$ defined by $s$ is also
smooth in this case. Thus condition (1) is satisfied.
Condition (2) follows from Theorems~\ref{t:1} and \ref{t:1.2}, so
the proof of Theorem \ref{fib} is now complete.

\section{Rational points on some del Pezzo surfaces of degrees $1$ and $2$}\label{s:DP}

Theorem \ref{t:1} can be used to study weak approximation 
for suitable del Pezzo surfaces of low degree, where our knowledge is 
still largely conditional. In this section we construct families of del Pezzo surfaces 
of degree $1$ and $2$ for which the failure of weak approximation
is controlled by the Brauer--Manin obstruction. 
Recall that a smooth and projective surface $V$ is called
{\it minimal} if any birational morphism $V\to V'$, where
$V'$ is also smooth and projective, is an isomorphism.
The surfaces that we construct will be
minimal, so our results do not follow from earlier results 
for del Pezzo surfaces of higher degree.

We start with describing del Pezzo surfaces to which
Theorem \ref{t:1} can be applied, in terms of orbits
of the Galois group action on the set of exceptional curves. 

Let $\Gamma_d$ be the graph whose
vertices are the exceptional curves on a del Pezzo surface of degree $d$
defined over an algebraically closed field; two vertices are
connected by $n$ edges if the intersection index of the corresponding
curves is $n$. A del Pezzo surface $X$ of degree $d$ defined over $\QQ$
induces an action of the Galois group 
$G={\rm Gal}(\bar \QQ/\QQ)$ on $\Gamma_d$ realised as the graph of 
exceptional curves on $\bar X=X\times_{\QQ}\bar \QQ$. 

Let $\Gamma(1)$ be the graph with two vertices joined by a single
edge. For a positive integer $r$ we denote by $\Gamma(r)$ 
the disconnected union of $r$ copies of $\Gamma(1)$. Recall that a
subgraph $\Gamma'$ of a graph $\Gamma$ is {\em induced} if the
vertices of $\Gamma'$ are connected by exactly the same edges as in $\Gamma$.

\begin{proposition} \label{p}
Consider the family of del Pezzo surfaces of degree $d\leq 7$ 
over $\QQ$ for which $\Gamma_d$ has an induced subgraph 
$\Gamma(8-d)$ such that all the connected components of $\Gamma(8-d)$ are 
$G$-invariant. All surfaces in this family have the property 
that the Brauer--Manin 
obstruction is the only obstruction to weak approximation.
Moreover, if $d\in\{1,2,4\}$ then the surfaces for which 
no vertex of $\Gamma(8-d)$ is fixed by $G$ are minimal over $\QQ$.
\end{proposition}
\begin{proof} Pick a connected component of $\Gamma(8-d)$, and let 
$C\in \Pic(\bar X)$ be the class of
the sum of corresponding exceptional curves. 
We have $(C,C)=0$, and this implies that $C$ is the class of 
a geometrically reducible fibre of a conic bundle
morphism $\pi:X\to \PP^1_\QQ$.
The curves orthogonal to $C$ under 
the intersection pairing are components of the fibres of $\pi$.
Thus the unions of exceptional
curves corresponding to the connected components of $\Gamma(8-d)$
give rise to $8-d$ degenerate fibres of $\pi$, which are all defined
over $\QQ$. A del Pezzo surface of degree $d$ which is a conic bundle 
has exactly $8-d$ degenerate fibres. Thus all the degenerate 
fibres of $\pi:X\to \PP^1_\QQ$ are defined over $\QQ$ and Theorem
\ref{t:1} applies to $X$. 

A conic bundle surface 
is called {\em relatively} minimal if all the fibres of 
the conic fibration are integral. 
For the surfaces considered in Theorem \ref{t:1}
this means that no component of a degenerate fibre is defined over $\QQ$,
or equivalently, no vertex of $\Gamma(8-d)$ is fixed by $G$.
By a theorem of Iskovskikh \cite[Thm.~4]{isk}, if a del Pezzo
surface of degree 1, 2 or 4 is a relatively minimal conic bundle,
then it is a minimal surface.
\end{proof}

Let $f, g, h\in \QQ[t]$ be polynomials such that 
$f(t)g(t)h(t)=c\prod_{i=1}^r(t-e_i)$, for 
$c\in \QQ^*$ and pairwise different
$e_1,\ldots,e_{r}\in \QQ$. 
Assume that $\ell=\deg f$, $m=\deg g$, $n=\deg h$
are integers of the same parity such that  $\ell\leq m\leq n$.
Consider the smooth surface in 
$\PP_\QQ^2\times \AA_\QQ^1$ defined by
\begin{equation}\label{eq:1}
f(t)x^2+g(t)y^2+h(t)z^2=0,
\end{equation}
where $t$ is a coordinate function on $\AA_\QQ^1$.
We embed $\AA_\QQ^1$ into $\PP_\QQ^1$ as the complement to the point $\infty$.
We may also take $\AA_\QQ^1\subset\PP_\QQ^1$ to be the
complement to the point $t=0$, with the coordinate function $T=1/t$.
Let $F(T)=T^{\ell} f(1/T)$, $G(T)=T^{m} g(1/T)$, $H(T)=T^{n} h(1/T)$,
and consider the smooth surface in $\PP_\QQ^2\times \AA_\QQ^1$
given by
\begin{equation}\label{eq:2}
F(T)X^2+G(T)Y^2+H(T)Z^2=0.
\end{equation}
Let $\pi:V\rightarrow \PP_\QQ^1$ be the conic bundle obtained by
gluing the surface (\ref{eq:1}) with the surface (\ref{eq:2}).
For this we
identify the restrictions of the two fibrations to
$\PP^1_\QQ\setminus\{0,\infty\}$ by means of the isomorphism
$t=T^{-1}$, $x=T^{\ell_1}X$, $y=T^{m_1}Y$, $z=T^{n_1}Z$,
where $(\ell,m,n)=2(\ell_1,m_1,n_1)$ or 
$(\ell,m,n)+(1,1,1)=2(\ell_1,m_1,n_1)$.
Since $F(0)G(0)H(0)\not=0$,
the fibre of $\pi$ at $t=\infty$ is smooth, so
$\pi$ has precisely $r=\ell+m+n$ degenerate  fibres.

Suppose $r=5$, with $(\ell,m,n)=(1,1,3)$. Setting $z=1$ in \eqref{eq:1}
and passing to homogeneous coordinates
we obtain a smooth cubic surface in $\PP^3_\QQ$ with the equation
$$c_1(u-e_1v)x^2+c_2(u-e_2v)y^2+c_3(u-e_3v)(u-e_4v)(u-e_5v)=0.$$
It contains the line $u=v=0$. If
the conic bundle is relatively minimal, then, contracting this line,
we obtain a minimal del Pezzo surface of degree 4 with a $\QQ$-point
by \cite[Prop.~2.1]{I2}.

Suppose next that $r=6$, with $(\ell,m,n)=(2,2,2)$. 

\begin{proposition}
Let $f(t)=a(t-e_1)(t-e_2)$, $g(t)=b(t-e_3)(t-e_4)$, $h(t)=c(t-e_5)(t-e_6)$,
where $e_1,\ldots,e_6\in\QQ$ are pairwise different, and 
$a,b,c\in\QQ^*$. If $f(t)$, $g(t)$ and $h(t)$
are linearly independent over $\QQ$, then 
$V$ is a del Pezzo surface of degree $2$ for which the Brauer--Manin 
obstruction is the only obstruction to weak approximation.
If, moreover, the classes
$$-1,\ a,\ b,\ c,\ e_i-e_j \ \text{for} \ 1\leq i<j\leq 6,$$ 
are linearly independent in the 
$\FF_2$-vector space $\QQ^*/\QQ^{*2}$, then $V$ is minimal. 
\end{proposition}
\begin{proof}
Let us write $f(t)=f_2t^2+f_1t+f_0$, with $f_0,f_1,f_2\in \QQ$,
and similarly for $g(t)$ and $h(t)$.
Using the equality of degrees of $f(t)$, $g(t)$ and $h(t)$ 
one checks immediately that
the projection to the first factor $\PP_\QQ^2\times \AA_\QQ^1
\to \PP_\QQ^2$ gives rise to a morphism $\varphi:V\to \PP^2_\QQ$. Linear
independence of $f(t)$, $g(t)$ and $h(t)$ implies that 
$\varphi$ has finite fibres, and so $V$ is a double covering of $\PP^2_\QQ$
ramified in the quartic curve
$$(f_1x^2+g_1y^2+h_1z^2)^2=4(f_0x^2+g_0y^2+h_0z^2)(f_2x^2+g_2y^2+h_2z^2).$$
This curve is smooth because $V$ is smooth. A double covering of $\PP^2_\QQ$
ramified in a smooth quartic is a del Pezzo surface of degree 2.
The first statement now follows from Theorem~\ref{t:1}. 
The final condition in the proposition implies that the conic bundle 
$\pi:V\rightarrow \PP_\QQ^1$
is relatively minimal, whence $V$  is minimal by Proposition \ref{p}.
\end{proof}

The case $r=7$ translates as $K_V^2=1$. This
arises if and only if $(\ell,m,n)=(1,1,5)$ or $(\ell,m,n)=(1,3,3)$.
We claim that 
neither of these surfaces can be isomorphic to a del Pezzo of degree 1. 
To see this we recall that del Pezzo surfaces are defined by 
the property that their anticanonical divisor is ample. 
It therefore suffices to find a geometrically integral curve $C$ on $V$ 
for which $(C,-K_V)\leq 0$.  To do so we adapt an argument of 
Iskovskikh  \cite[Prop.~1.3 and Cor.~1.4]{I2}.
In the case $(\ell,m,n)=(1,1,5)$ consider the curve $C$ which is 
the Zariski closure in $V$ of the closed subset of \eqref{eq:1} 
given by $z = 0$. We claim that this is a smooth curve of genus 
$0$ such that $(C, - K_V ) = -1$.
To see this we note that 
$C$ is a smooth curve of genus 0 such that $(C,F)=2$,
where $F\in \Pic(V)$ is the class of a fibre. 
The divisor of the rational function $z/x$ on $V$ is $C+2F_\infty-C'$,
where $F_\infty$ is the fibre at infinity and $C'$
is the Zariski closure in $V$ of the closed subset of \eqref{eq:1}
given by $x=0$. Since $(C,C')=1$ we see that $(C^2)=-3$, and then
from the adjunction formula we find that $(C,-K_V)=-1$, as claimed.
In the case $(\ell,m,n)=(1,3,3)$ we consider the pencil of genus 1 curves
$E=E_{(\lambda:\mu)}$ 
cut out by $\lambda y+\mu z=0$ on $V$. 
It is easy to see
that $(E,E)=1$, and hence adjunction gives $(E,-K_V)=1$.
It follows that $E=-K_V$. This pencil contains two reducible members, 
each consisting of the union of one component of the degenerate 
fibre at $f(t)=0$ and a residual rational curve $C$. It follows
that $(C,-K_V)=0$.

We can use some special conic bundles with {\em eight} degenerate fibres
to construct del Pezzo surfaces of degree 1 to which
Theorem \ref{t:1} can be applied. 
Note that $r=8$ gives $K_V^2=0$.
Let $e_1,\ldots,e_8\in \QQ$ be pairwise distinct, and let
$c_1,\,c_2\in\QQ^*$. 
Let $\pi:V\to \PP^1_\QQ$ be the conic bundle constructed as above
from the surface given by the equation
\begin{equation}
x^2=c_1^2\prod_{i=1}^4\frac{t-e_i}{e_8-e_i}y^2+c_2^2\prod_{j=5}^8(t-e_j)z^2
\label{eee}
\end{equation}
in $\PP^2_\QQ\times\AA^1_\QQ$. This conic bundle is not
relatively minimal because the fibre at $t=e_8$ is a union of
two components defined over $\QQ$. Each of them
can be smoothly contracted, thus 
producing a conic bundle surface $W\to \PP^1_\QQ$ 
with seven degenerate fibres. 

Recall that the discriminant of the quartic polynomial
$p(t)=\sum_{i=0}^4p_it^i$ is a homogeneous form 
$D_4(p_0,\ldots,p_4)$ of degree 6.
Thus $D_4=0$ defines a hypersurface $Z\subset \PP^4_\QQ$ of degree 6.
The space of projective lines in $\PP^4_\QQ$ is naturally identified 
with the Grassmannian $Gr(2,5)$.
The open subset of $Gr(2,5)$ parameterising those lines which meet $Z$
in six distinct complex points is non-empty. 
Let $\AA^5_\QQ$ be the space of 
polynomials of degree at most 4. 
Joining two points by a line gives a dominant 
rational map from $\AA^5_\QQ\times \AA^5_\QQ$ to $Gr(2,5)$. 
It follows that the open subset of $\AA^5_\QQ\times \AA^5_\QQ$
consisting of pairs of polynomials $(p(t),q(t))$ 
such that the discriminant of 
$rp(t)+sq(t)$ 
vanishes for exactly six points
$(r:s)\in \PP^1_\CC(\CC)$, is non-empty. 
These six points of $Z$ are necessarily smooth in $Z$, 
and hence for each of them $rp(t)+sq(t)$ has exactly one double root.
We conclude that there is a non-zero
polynomial $f(p_0,\ldots,p_4,q_0,\ldots,q_4)$ 
with coefficients in $\QQ$ such that if $f(p_0,\ldots,p_4,q_0,\ldots,q_4)\not=0$,
then $rp(t)+sq(t)$ has multiple roots for exactly six values of 
$(r:s)\in \PP^1_\CC(\CC)$, and for each of these values 
$rp(t)+sq(t)$ has exactly one double root.
Writing the coefficients as symmetric functions
of the roots, and applying this to the polynomials
$$p(t)=c_1^2\prod_{i=1}^4\frac{t-e_i}{e_8-e_i}
\quad\text{and}\quad q(t)=c_2^2\prod_{j=5}^8(t-e_j)$$
we obtain a non-zero polynomial $F(e_1,\ldots,e_8,c_1,c_2)$ 
with coefficients in $\QQ$.

\begin{proposition} 
If $e_1,\ldots,e_8\in \QQ$ and $c_1, c_2\in \QQ^*$ 
satisfy 
$F(e_1,\ldots,e_8,c_1,c_2)\not=0$, 
then $W$ is a del Pezzo surface of degree $1$
over $\QQ$ for which the Brauer--Manin obstruction is the only
obstruction to weak approximation.
If, moreover, the classes of  
$e_i-e_j$, where $1\leq i\leq 4$ and $5\leq j\leq 8$,
are linearly independent in the 
$\FF_2$-vector space $\QQ^*/\QQ^{*2}$, then $W$ is minimal.
\end{proposition}
\begin{proof} For $(\lambda:\mu)\in
\PP^1_\QQ(\QQ)$ let $E_{(\lambda:\mu)}\subset V$ be the 
Zariski closure of the subset of (\ref{eee}) given by
$\lambda y+\mu z=0$. It has an affine equation
$u^2=\mu^2 p(t)+\lambda^2 q(t)$,
where $u=\lambda x/z$. Since $\deg p(t)=\deg q(t)=4$,
the smooth curves in this family have genus 1.
Let $E$ be the class of $E_{(\lambda:\mu)}$ in $\Pic(\bar V)$.
Since $E_{(1:0)}$ and $E_{(0:1)}$ are disjoint, we have
$(E,E)=0$. Thus $V$ is an elliptic surface, with  a morphism 
$\ee:V\to\PP^1_\QQ$ such that the fibre above $(\lambda:\mu)$ is
$E_{(\lambda:\mu)}$. 

As was explained above, the condition $F(e_1,\ldots,e_8,c_1,c_2)\not=0$ 
guarantees that there are exactly six points $(r:s)\in\PP^1_\CC(\CC)$ 
such that $r p(t)+s q(t)$ is not separable. Moreover, for each of these
values of $(r:s)$ this polynomial has exactly one double root.
Since $p(t)$ and $q(t)$ are separable,
$\ee:V\to\PP^1_\QQ$ has exactly twelve singular geometric fibres, 
and each of them is a geometrically irreducible rational curve with one node.
In particular, the fibres of $\ee:V\to\PP^1_\QQ$ do not contain 
exceptional curves, that is,
smooth rational curves with self-intersection $-1$. In addition,  
the elliptic surface $V$ is geometrically rational so we have
$-K_V=E$ by \cite[Cor.~12.3, p.~214]{BHPV}. 
It follows that if $C\subset \bar V$ is an irreducible curve which is not
a fibre of $\ee$, then $(-K_V,C)>0$.

Let $L$ be the irreducible component of the
fibre of $\pi:V\to \PP^1_\QQ$ at $t=e_8$ which is contracted
to a point on $W$. Since $(L,E)=\frac{1}{2}(F,E)=1$, 
we see that $L$ is a section of $\ee:V\to\PP^1_\QQ$.
Let $\sigma: V\to W$ be the contraction of $L$.
It is easy to see that $-K_W=\sigma_*(E)$. Any two distinct
curves $\sigma(E_{(\lambda:\mu)})$ and $\sigma(E_{(\lambda':\mu')})$
have exactly one common point $\sigma(L)$, whence $(K_W,K_W)=1$. 
It follows that every irreducible curve in $\bar W$ has positive
intersection with $-K_W$. By the Nakai--Moishezon criterion
$-K_W$ is ample, and so $W$ is a del Pezzo surface of degree 1.
Theorem \ref{t:1} can be applied to the conic bundle 
$W\to\PP^1_\QQ$. This proves our first statement.

The components of the degenerate fibre at $e_i$, for $i=1,\,2,\,3,\,4$,
are defined over $\QQ(\sqrt{a_i})$, where $a_i=\prod_{j=5}^8(e_i-e_j)$.
The components of the degenerate fibre at $e_j$, for $j=5,\,6,\,7$,
are defined over $\QQ(\sqrt{a_j})$, where 
$a_j=\prod_{i=1}^4(e_j-e_i)/(e_8-e_i)$. Now the condition
in the last sentence of the proposition implies the relative
minimality of the conic bundle $W\to \PP^1_\QQ$, and hence,
by Proposition \ref{p}, 
the minimality of the del Pezzo surface $W$ of degree 1.
\end{proof}


\begin{thebibliography}{cc}

\bibitem{BHPV} W. Barth, K. Hulek, C. Peters and A. van de Ven.
{\em Complex compact surfaces}. 2nd ed., Ergebnisse der Mathematik und 
Ihrer Grenzgebiete, Springer-Verlag, 2004.

\bibitem{ct-4}
J.-L. Colliot-Th\'el\`ene, Surfaces rationnelles fibr\'ees en coniques
de degr\'e $4$.
{\em S\'eminaire de th\'eorie des nombres, Paris 1988--1989},  43--55,
Progr.\ Math. {\bf  91}, Birkh\"auser, 1990.

\bibitem{ct-s-79} 
J-L. Colliot-Th\'el\`ene and J-J. Sansuc,
La descente sur les vari\'et\'es rationnelles. In {\em Journ\'ees de 
g\'eom\'etrie alg\'ebrique d'Angers (Juillet 1979)}, 223--237, 
Sijthoff \& Noordhoff, 1980. 

\bibitem{ct-s-82}
J-L. Colliot-Th\'el\`ene and J-J. Sansuc,
Sur le principe de Hasse et l'approximation faible, et sur une 
hypoth\`ese de Schinzel. {\em Acta Arith.}  {\bf 41} (1982), 33--53.

\bibitem{ct-s-86}
J-L. Colliot-Th\'el\`ene and J-J. Sansuc,
La descente sur les surfaces rationnelles fibr\'ees en coniques.
{\em C.\ R.\ Acad.\ Sci.\ Paris} {\bf 303} (1986),  303--306. 


\bibitem{ct-s-87}
J-L. Colliot-Th\'el\`ene and J-J. Sansuc,
La descente sur les vari\'et\'es  rationnelles, II.
{\em Duke Math.\ J.} {\bf 54} (1987),  375--492. 

\bibitem{CTSk} J-L. Colliot-Th\'el\`ene and A.N. Skorobogatov.
Descent on fibrations over $\PP^1_k$ revisited.
{\em Math. Proc. Camb. Phil. Soc.} {\bf 128} (2000), 383--393.

\bibitem{ct-swd-94}
J-L. Colliot-Th\'el\`ene and P. Swinnerton-Dyer, 
Hasse principle and weak approximation for pencils of Severi--Brauer 
and similar varieties. {\em 
J.\ reine angew.\ Math.} {\bf 453} (1994), 49--112.


\bibitem{ct-c-s} 
J.-L. Colliot-Th\'el\`ene, D. Coray and J.-J. Sansuc, 
Descente et principe de Hasse pour certaines vari\'et\'es rationnelles.
{\em J.\ reine angew.\ Math.}  {\bf 320}  (1980), 150--191.

\bibitem{CT}
J-L. Colliot-Th\'el\`ene, J-J. Sansuc and P. Swinnerton-Dyer, 
Intersections of two quadrics and Ch\^atelet surfaces, I.  
{\em J.\ reine angew.\ Math.} {\bf 373}  (1987), 37--107;  
II. ibid. {\bf 374}  (1987), 72--168. 

\bibitem{GT}
B. Green and T. Tao, Linear equations in primes.
{\em Annals of Math.} {\bf 171} (2010),  1753--1850.

\bibitem{GTZ}
B. Green, T. Tao and T. Ziegler, 
An inverse theorem for the Gowers $U_{s+1}[N]$-norm.
{\em Annals of Math.} {\bf 176} (2012),  1231--1372. 

\bibitem{harari}
D. Harari, M\'ethode des fibrations et obstruction de Manin.
{\em Duke Math. J.} {\bf 75} (1994), 221--260.


\bibitem{I2} V.A. Iskovskikh, Rational surfaces with a sheaf of 
rational 
curves and with a positive square of canonical class. (Russian) 
{\em Mat.\ Sb.} {\bf 83} (1970), 90--119. English translation:
{\em Math. USSR-Sbornik} {\bf 12} (1970), 91--117.

\bibitem{isk} V.A. Iskovskikh, 
Minimal models of rational surfaces over arbitrary fields. (Russian)
{\em Izvestiya Akad. Nauk SSSR Ser. Mat.} {\bf 43} (1979), 19--43. 
English translation: {\em Math.\ USSR Izv.} {\bf 14} (1980), 17--39. 


\bibitem{lm} L. Matthiesen,
Linear correlations amongst numbers represented by positive definite 
binary quadratic forms.
{\em Acta Arith.} {\bf 154} (2012), 235--306. 


\bibitem{lm'} L. Matthiesen,
Correlations of representation functions of binary quadratic forms.
{\em Acta Arith.} {\bf 158} (2013), 245--252.

\bibitem{s-conic}
P. Salberger, 
Sur l'arithm\'etique de certaines surfaces de Del Pezzo. {\em
C.\ R.\ Acad.\ Sci.\ Paris} {\bf 303} (1986), 273--276. 


\bibitem{salb-skoro}
P. Salberger and A.N. Skorobogatov, Weak approximation for surfaces 
defined by two quadratic forms. {\em Duke Math.\ J.} {\bf 63} (1991), 517--536.


\bibitem{serre}
J.-P. Serre, {\em Cohomologie Galoisienne}. 5th ed., SLM {\bf 5}, 
Springer-Verlag, 1994.

\bibitem{Sk}
A.N. Skorobogatov, 
Arithmetic on certain quadric bundles of relative dimension 2. I.
{\em J. reine  angew.\ Math.} {\bf 407} (1990),  57--74.


\bibitem{S96}
A.N. Skorobogatov, Descent over the fibrations over the projective line.
{\em Amer. J. Math.} {\bf 118} (1996), 905--923.


\bibitem{skoro}
A.N.  Skorobogatov, {\em Torsors and rational points}. 
Cambridge University Press, 2001.


\bibitem{swd-94}
P. Swinnerton-Dyer, 
Rational points on pencils of conics and on pencils of quadrics. 
{\em J.\ London Math.\ Soc.} {\bf 50} (1994), 231--242.


\bibitem{swd-toulouse}
P. Swinnerton-Dyer,  Rational points on some pencils of conics 
with 6 singular fibres. {\em  Ann.\ Fac.\ Sci.\ Toulouse Math.} 
{\bf 8} (1999),  331--334.

\end{thebibliography}
\end{document}